\documentclass[]{amsart}
\usepackage{ amsfonts, amsmath, amsthm, amstext, amssymb, amscd, mathrsfs, verbatim, color, enumerate, lscape,extarrows}
\usepackage{graphicx}

\theoremstyle{plain}
\newtheorem{theorem}{Theorem}[section]

\newtheorem{lemma}[theorem]{Lemma}
\theoremstyle{definition}

\newtheorem{proof*}{Proof.}
\numberwithin{equation}{section}
\newtheorem{remark}[theorem]{Remark}

\DeclareMathAlphabet{\mathpzc}{OT1}{pzc}{m}{it}
\newtheorem*{acknowledgment}{Acknowledgment}

\newcommand{\arc}[1]{\overline{#1} }

\newcommand{\A}{{\mathcal A}} 
\newcommand{\B}{{\mathcal B}} 
\newcommand{\C}{{\mathcal C}} 
\newcommand{\F}{{\mathcal F}} 
\newcommand{\h}{{\mathcal H}} 
\newcommand{\I}{{\mathcal I}} 
\newcommand{\K}{{\mathcal K}} 
\renewcommand{\L}{{\mathcal L}} 
\newcommand{\m}{{\mathcal S}} 
\renewcommand{\P}{{\mathcal P}} 
\newcommand{\Q}{{\mathcal Q}} 
\newcommand{\R}{{\mathcal R}} 
\renewcommand{\S}{{\mathscr S}} 
\newcommand{\s}[1]{\mathcal S^{*}_{#1} }
\newcommand{\U}{{\mathcal U}} 

\newcommand{\PG}{\mbox{\rm PG}}

\date{\today}

\author{Alice M. W. HUI} \thanks{Current address: Department of Mathematics, The Chinese University of Hong Kong}

\title[A geometric proof of a Wilbrink's characterization]{A geometric proof of Wilbrink's characterization of even order classical unitals}

\begin{document}

\begin{abstract}
Using geometric methods and without invoking deep results from group theory, we prove that a classical unital of even order $n\geq4$ is characterized by two conditions (I) and (II): (I) is the absence of O'Nan configurations of four distinct lines intersecting in exactly six distinct points; (II) is a notion of parallelism. This was previously proven by Wilbrink (1983), where the proof depends on the classification of finite groups with a split BN-pair of rank 1.
\end{abstract}

\maketitle

\noindent
Keywords: unital; classical unital; Hermitian curve; spread\\
MSC(2000): 51E20, 05B25, 05B25, 51E21, 51E23\\

\section{Introduction}\label{st 1}
\noindent
A unital of order $n>2$ is a design with parameters $2$--$(n^3 + 1, n + 1, 1)$ (see \cite{Dem2, BE}). If $\pi$ is a projective plane of order $m$, i.e. a $2$--$(m^2 + m + 1, m + 1, 1)$ design, and if a unital $U$ is an induced substructure of $\pi$, then we call $U$ an {\it embedded unital}.
Some embedded unitals $U$ of order $n$ are the incidence structure formed from the absolute points and non-absolute lines of a unitary polarity in a projective plane $\pi$ of order $n^2$.
Any unital which is isomorphic to such a unital $U$ as a design is called a {\it polar unital}. Further if the ambient plane is $\PG(2,n^2)$, then the unital is called {\it classical}.
The set of absolute points of a unitary polarity in $\PG(2,n^2)$ is called the {\it Hermitian curve} (see \cite{Hir1, HP}).

In 1972 \cite{ON}, O'Nan showed that the classical unital does not contain a configuration of four lines meeting in six points (an {\it O'Nan configuration}). Piper (1981) \cite{Pi} conjectured that this property characterizes the classical unital. Wilbrink (1983) \cite{Wil} characterized the classical unital by three conditions (I), (II) and (III). His proof depends on a result in the classification of finite groups with a split BN-pair of rank 1. Wilbrink \cite{Wil} further proved that when the order of unital is even, (III) is a necessary condition of (I) and (II).

Let $\U$ be a unital of even order $n \geq 4$ satisfying Wilbrink's conditions (I) and (II). In this article, we give an alternative proof that $\U$ is classical without invoking deep results from group theory, as follows. We construct from $\U$ a hyperbolic Buekenhout unital $\U'$ \cite{Bkt} in $\PG(2, n^2)$, via the Bruck-Bose construction of projective plane \cite{BB1,BB2}. Then we prove that $\U$ is isomorphic to $\U'$, and hence is classical by a result of Barwick \cite{Bar}. To construct $\U'$, we shall consider some inversive planes and a generalized quadrangle derived from $\U$ (Wilbrink \cite{Wil}), and the special spreads of $\U$ (Hui and Wong \cite{HW2}). We also need a theorem of Cameron and Knarr \cite{CK} on how to build a regular spread of $\PG(3,q)$ from a tube in $\PG(3,q)$, and a theorem of Hui \cite{H} on when two inversive planes are identical.

In Section \ref{st 2}, we follow Wilbrink's \cite{Wil} construction of the inversive planes $\I(x)$ at each point $x$ of $\U$. Then following the work in \cite{HW2}, we construct a special spread $\mathcal S_L$ for each line $L$ of $\U$ using these inversive planes. As a consequence, $\U$ can be embedded in a projective plane $\pi$ as a polar unital \cite{HW2}. This enables us to define self-polar triangles with respect to $\U$ intrinsically in terms of Wilbrink's $x$-parallelism (Theorem \ref{thm LMN}).

In Section \ref{st 3}, by studying the inversive planes $\I(x)$ for various $x$'s, we prove that the set of points of $\U$ is partitioned into a self-polar triangle, and $n - 2$ subsets of $(n+1)^2$ points triply ruled by lines through the vertices of the triangle (Theorem \ref{thm bigP}). This describes how unital lines in $\pi$ through distinct non-unital points intersect.

In Section \ref{st 4}, we fix one line $L$ of $\U$ and consider the generalized quadrangle $GQ(L)$ as in Wilbrink \cite{Wil}. Through $GQ(L)$, we associate $\U$ with $Q(4,n)$ formed by the set of points and lines of a parabolic quadric $\P$ in $\PG(4,n)$ \cite{PT}. Wilbrink's construction gives naturally a 3-dimensional subspace $\Sigma$ of $\PG(4,n)$. We find a spread $\S$ in $\Sigma$ by studying the special spread $\mathcal S_L$. We then prove that $\S$ is regular (Theorem \ref{thm regular}) using a result on tubes in $\PG(3,n)$ (Cameron and Knarr \cite{CK}) and properties of self-polar triangles with respect to $\U$.

In Section \ref{st 5}, we prove that the partition of $\U$ into a self-polar triangle and triply ruled sets corresponds to a pencil of quadrics in $\Sigma$ of two lines and $n - 1$ hyperbolic quadrics (Theorem \ref{thm pencil}). In particular, this gives a correspondence between the structure of $\U$ and that of $\Sigma$. The regularity of $\S$ is essential for proving Theorem \ref{thm pencil}, because we have to describe the reguli of $\S$ in terms by geometry of $\U$ by applying a result of Hui \cite{HW2} to the Miquelian inversive plane formed by the lines and reguli of $\S$ (Bruck \cite{Bru}).

In Section \ref{st 6}, by considering the spread $\S$ of $\Sigma$ in $\PG(4,n)$, we construct a projective plane $\overline{\pi(\S)}$ by the Bruck-Bose construction \cite{BB1}. Since $\S$ is regular, $\pi(\S)$ is $\PG(2,n^2)$ \cite{BB2}.
By Buekenhout \cite{Bkt}, $\P$ defines a hyperbolic Buekenhout unital $\U'$ in $\overline{ \pi(\S)}$ \cite{Bkt} (also known as a nonsingular Buekenhout unital). By Barwick \cite{Bar}, since $\overline{\pi(\S)}\cong\PG(2,n^2)$, $\U'$ is the classical unital. With the help of Theorem \ref{thm pencil}, we write down an isomorphism between $\U$ and the classical unital $\U'$ (Theorem \ref{thm main}).

\section{Self-polar triangles and Parallelism}\label{st 2}
\noindent
Let $\U$ be a unital of even order $n \geq 4$, satisfying Wilbrink's first two conditions (I) and (II) \cite{Wil}:
\begin{enumerate}
\item[(I)] $\U$ contains no O'Nan configurations.
\item[(II)] Let $x$ be a point, $L$ be a line through $x$, and $M$ be a line missing $x$, such that $L$ and $M$ meets. For any point $y'\in L \setminus \{x\}$, there is a line $M'$ through $y'$ but not $x$ meeting all lines from $x$ which meet $M$.
\end{enumerate}

\noindent Following Wilbrink \cite{Wil}, we introduce {\it $x$-parallelism} in $\U$ \cite{Wil}:
\begin{itemize}
\item[]Let $x$ be a point, and $M,M'$ be two lines missing $x$. $M,M'$ are said to be $x$-parallel if $M,M'$ intersect the same lines through $x$. We write $M\|_x M'$.
\end{itemize}

\noindent $\|_x$ defines an equivalence relation on the set of all lines missing $x$ \cite{Wil}. We denote the equivalence class of a line $M$ under $\|_x$ by $\arc{M}^x$, or simply $\arc{M}$ if there is no confusion.

Further following Wilbrink \cite{Wil}, we introduce an inversive plane $\I(x)$ of order $n$ for every point $x$ in $\U$ (\cite{Wil}, Lemmas 1, 2 and Corollary 3; see also \cite{HW2}). The points of $\I(x)$ are the lines of $\U$ through $x$ together with a symbol $\infty_x$. The circle set of $\I(x)$ is $\mathcal C^x \cup \mathcal C_x$. Here $\mathcal C^x$, $ \mathcal C_x$ are given by:
$\mathcal C^x$ is the set of $\|_x$-equivalence classes on the set of lines of $\U$ missing $x$;
$\mathcal C_x$ consists of blocks of the form $C_x(L, L') \cup \{\infty_x\}$, where for any lines $L,L'$ on $x$, $C_x(L,L')$ = $\{L,L'\} \cup \{L'' \mid L''$ is a line through $x$ such that no line of through $x$ meets $L, L'$ and $L''\}$.
The incidence in $\I(x)$ is defined as follows:
Whenever $L$ is a line of $\U$ through $x$ and $M$ is a line of $\U$ missing $x$, $L$ is incident with $\arc M$ in $\I(x)$ if and only if $L$ meets $M$ in $\U$;
whenever $L, L', L''$ are lines of $\U$ through $x$, $L$ is incident with $C_x(L', L'')$ in $\I(x)$ if and only if $L \in C_x (L',L'')$;
the point $\infty_x$ is incident with all circles in $\C^x$ but none in $\C_x$.

According to Dembowski \cite{Dem1}, $\I(x)$ is egglike. By Thas \cite{Thas1}, every flock in $\I(x)$ is uniquely determined by its carriers. We denote the unique flock in $\I(x)$ with carriers $p_1$ and $p_2$ by $\F(p_1,p_2)$.

With the help of flocks of the form $\F(L,\infty_x)$, $\U$ can be shown to satisfy condition ($P$) \cite[Theorem 1.6]{HW2}, formulated in terms of {\it special spreads} in $\U$.

Let $\m$ be a spread of $\U$, with $L \in \m$. Then $\m$ is {\it special with respect to $L$} if the following condition is satisfied:
\begin{itemize}
\item[ ]
for any point $x$ on $L$, $\m \setminus \{ L\}$ can be partitioned into $n - 1$ subsets $\L_1^x, \cdots, \L_{n - 1}^x$, each of cardinality $n$, and the set of lines on $x$, except $L$, can be partitioned into $n - 1$ subsets $\K_x^1, \cdots, \K_x^{n - 1}$, each of cardinality $n + 1$, such that whenever $L' \in \L_i^x$ and $K \in \K_x^j$, $L'$ and $K$ intersect if and only if $i = j$.
\end{itemize}

\noindent Now, \cite[Theorem 1.6]{HW2} says that $\U$ satisfies condition ($P$), which is a strengthened version of condition ($p$):
\begin{enumerate}
\item[($p$)] Let $L_1, L_2, \cdots, L_{n^4-n^3 + n^2}$ be the lines of $\U$. There exists a family of lines $\mathscr F = \{ \m_{L_1}$, $\m_{L_2}$, $\cdots$, $\m_{L_{n^4-n^3 + n^2}} \}$ such that:
\begin{enumerate}
 \item[(i)] For $i = 1, 2, \cdots, n^4-n^3 + n^2$, $\m_{L_i}$ is a spread containing $L_i$.
 \item[(ii)] For $i\neq j$, $L_i \in \m_{L_j} \setminus \{L_j\}$ if and only if $L_j \in \m_{L_i} \setminus \{L_i\}$.
 \item[(iii)] For any two lines $L_i$ and $L_j$ missing each other, there exists a line $L_k$ such that $ L_k \in \m_{L_i} \setminus \{L_i\}$ and $ L_k \in \m_{L_j} \setminus \{L_j\}$.
\end{enumerate}
\item[($P$)] ($p$) holds such that for $i = 1,2,\cdots,n^4-n^3 + n^2$, $\m_{L_i}$ is a special spread with respect to $L_i$.
\end{enumerate}

\noindent In the above statement, the set $\m_L$ is given explicitly by this construction:
\begin{itemize}
\item[] Pick some point $x$ on $L$. $\m_L$ is given by $\arc{L_1}$ $\cup \arc{L_2}$ $\cup \cdots$ $\cup \arc{L_{n - 1}}$ $\cup \{L\}$, where $\arc{L_1}$, $\arc{L_2}$, $\cdots$, $\arc{L_{n - 1}}$ are the circles of $\F(L, \infty_{x})$ in $\I(x)$.
\end{itemize}
$\m_L$ is shown to be independent of $x$ \cite[Lemma 5.4]{HW2}.

For each line $L$ of $\U$, denote by $\s{L}$ the set $\m_L \setminus \{L\}$.

By \cite[Theorem 1.1]{HW2}, since $\U$ satisfies ($p$), $\U$ can be embedded in a projective plane $\pi$ as a polar unital, so that in $\pi$, for each unital line $J$, the unital lines of $\pi$ through the pole of $J$ are exactly the lines in $\s{J}$. Since two distinct points in $\pi$ determine a unique line, $\s{J}\cap \s{J'}$ contains at most one (unital) line for any distinct unital lines $J,J'$. Thus there are at most three lines in $\m_{J}\cap \m_{J'}$.
Whenever $L,M,N$ are three distinct lines of $\U$ satisfying $\{L,M,N\}=\m_{L}\cap \m_{M} = \m_{L}\cap \m_{N} =\m_{M}\cap \m_{N}$, we say that $L,M,N$ form a {\it self-polar triangle with respect to $\U$}.

\begin{lemma}\label{lemma polartri}
Let $L,M$ be disjoint lines of $\U$. If $M\in \s{L}$ (or equivalently $L \in \s{M}$), then there exists a unique line $N$ of $\U$ such that $L,M,N$ form a self-polar triangle with respect to $\U$.
\end{lemma}

\begin{proof}
In $\pi$, $L$ and $M$ meet in a unique non-unital point $a$. Let $N$ be the polar line of $a$.
By construction of $\pi$, $L,M\in \s{N}$ and $N\neq L,M$.
Since $\U$ satisfies ($p$), $N\in \m_{L} \cap \m_{M}$ by ($p$)(ii).
By ($p$)(i), $L \in \m_{L}$, $M \in \m_{M}$ and $N \in \m_{N}$.
The result follows from the fact the there are at most three lines in $\m_{L}\cap \m_{M}$, $\m_{L}\cap \m_{N}$ and $\m_{M}\cap \m_{N}$.
\end{proof}

\begin{theorem}\label{thm LMN}
Let $L,M,N$ be disjoint lines of $\U$. Then the following are equivalent.
\begin{enumerate}
 \item \label{item LMN1}$L,M,N$ form a self-polar triangle with respect to $\U$.
 \item \label{item LMN2}$L \|_z M$ for any point $z \in N$.
 \item \label{item LMN4} Any line meeting two of $L,M,N$ meet all of $L,M,N$.
 \item \label{item LMN3}$M \in \s{L}$, and there exist distinct points $z_1,z_2 \in N$ such that $L \|_{z_1} M$ and $L \|_{z_2} M$.
\end{enumerate}
\end{theorem}

\begin{remark}
Since lines in a self-polar triangle play the same role, in Theorem \ref{thm LMN}, statement \eqref{item LMN1} is also equivalent to \eqref{item LMN2}$'$: $M \|_z N$ for any point $z \in L$, or other statement obtained by permuting $L,M,N$ in \eqref{item LMN2} and \eqref{item LMN3}.
\end{remark}

\begin{proof}[Proof of Theorem \ref{thm LMN}]
Let $x$ be a point on $L$, and $K_1, K_2, \cdots, K_{n + 1}$ be the lines through $x$ meeting $M$.

\eqref{item LMN1} $\Rightarrow$ \eqref{item LMN2}: Since $M\in \s{L}$, there is a point $y_i \in K_i$ such that $L \|_{y_i} M$ for $i = 1, 2, \cdots, n + 1$, by Lemmas 5.3 and 5.4 of \cite{HW2}. Hence, for each $i$, $\arc L = \arc M$ in $\I(y_i)$. Then there is a unique point $N_i$ such that $\arc L \in \F (N_i,\infty_{y_i})$. This $N_i$ is a line of $\U$ through $y_i$ such that $L\in \s{N_i} $ and $M \in \s{N_i}$. By condition ($p$), $N_i \in\s{L} \cap \s{M}$. Since there is at least one line in $\s{L} \cap \s{M}$, we have $N_1 = N_2 = \cdots = N_{n + 1}=N$.

\eqref{item LMN2} $\Rightarrow$ \eqref{item LMN4}:
For each point $z\in N$, since $L \|_z M$, there are $n+1$ lines through $z$ meeting $L$ and $M$. Thus there are $(n+1)^2$ lines meeting $L,M,N$. They are the lines meeting at least two of $L,M,N$.

\eqref{item LMN4} $\Rightarrow$ \eqref{item LMN2} follows from the definition of $z$-parallelism.

\eqref{item LMN4} $\Rightarrow$ \eqref{item LMN3}:
It suffices to show $M \in \s{L}$.
Let $z'\in N$. By \eqref{item LMN4}, any line through $z'$ meeting $M$ meets $L$. Thus, $L\|_{z'} M$.
By \eqref{item LMN4}, $K_i$ meets $N$ for $i=1,2,\cdots,n+1$. Thus $z'$ is the point on $K_i$ such that $L\|_{z'} M$.
By Lemmas 5.3 and 5.4 of \cite{HW2}, $M \in \s{L}$.

\eqref{item LMN3} $\Rightarrow$ \eqref{item LMN1}:
Let $N'$ be a line such that $L,M,N'$ form a self-polar triangle. Suppose $z_1\notin N'$. Let $x_1,x_2 \in L$ be distinct points. For $i = 1,2$, let $J_i$ be the line passing through $x_i$ and $z_1$. Then $J_i$ meets $M$ by \eqref{item LMN3}. Since \eqref{item LMN1} implies \eqref{item LMN4}, $J_1$ meets $N'$ at a point, say $w$. Since \eqref{item LMN1} implies \eqref{item LMN2}, we have $L \|_w M$. Thus, the four lines $J_1$, $J_2$, $M$, $w.x_2$ form an O'Nan configuration, which is a contradiction. ($w.x_2$ denotes the line through $w$ and $x_2$.) Hence $z_1 \in N'$. Similarly, $z_2 \in N'$. So $N=N'$.
\end{proof}

With Theorem \ref{thm LMN}, we characterize $\s{M}$ by $z$-parallelism.

\begin{lemma}\label{lemma LJzparallel}
Let $L,M$ be distinct lines of $\U$. Suppose $M\in \s{L}$. Then
$$\s{M}=\{J \mid J \mbox{ is a line of $\U$ such that there is a point $y\in M$ such that }L \|_y J\}.$$
\end{lemma}

\begin{proof}
Let $J$ be a line. Suppose there is a point $y\in M$ such that $L \|_y J$. Since $L \in \s{M}$ by condition ($p$) and $L \|_y J$, we have $\arc J = \arc L \in \F(M,\infty_y)$ in $\I(y)$. Thus, $J \in \s{M}$ by the construction of $\s{M}$.

To prove the reverse inclusion, it suffices to show that there are $|\s{M}| = n^2 - n $ $J$'s such that $L \| _z J$ for some $z \in M$. Let $N$ be the line such that $L,M,N$ form a self-polar triangle. For each point $z$ on $M$, there are exactly $n$ lines $z$-parallel to $L$. By Theorem \ref{thm LMN}, $L$ and $N$ are two of these $n$ lines.
Apart from $L$ and $N$, no line is $z$-parallel to $L$ for distinct $z$'s on $M$ by Theorem \ref{thm LMN}.
Hence, there are $2 + (n + 1) (n - 2) = n^2 - n$ lines $z$-parallel to $L$ for some $z$ on $M$, as desired.
\end{proof}

\section{Partition $\U$ into a self-polar triangle and triply ruled sets}\label{st 3}
\noindent
In this section, we prove that $\U$ can be partitioned into a self-polar triangle and $n-2$ triply ruled sets of $(n+1)^2$ points (Theorem \ref{thm bigP}).
This result will be used at the end of Section \ref{st 5}, where we relate the partition to a pencil of quadrics in a projective space.

 In Dover \cite[Theorem 3.2]{Dov} (see also Section 4 of Baker et al. \cite{BEKS}), it is proved that any classical unital
 admits such a partition
 by considering coordinates and its automorphism group. The argument in this section will yield a synthetic proof for Dover's result once we prove that $\U$ is classical in Section \ref{st 6}.

We describe how lines of $\m_L$ correspond to circles tangent to $\arc L$ in $\I(y)$, whenever $y$ is not a point on $L$.

\begin{lemma}\label{lemma SLinIy}
Let $L$ be a line of $\U$, and $y$ be a point not on $L$. Let $M \in \m_L$ be the line of $\U$ through $y$. Let $N$ be the line of $\U$ such that $L,M,N$ form a self-polar triangle. Then the following statements hold:
\begin{enumerate}
	\item \label{itemSLinIy1} For every $J \in \m_L \setminus \{L,M,N\}$, $\arc {J}$ is tangent to $\arc L$ in $\I(y)$.
	\item \label{itemSLinIy2} Every circle of type $\C^y$ tangent to $\arc L$ in $\I(y)$ is $\arc {J'}$ for a unique line $J'$ in $\m_L \setminus \{L,M,N\}$.
\end{enumerate}
\end{lemma}

\begin{proof}
(1) Let $J \in \m_L \setminus \{L,M,N\}$.
It suffices to show that there is a unique line of $\U$ through $y$ meeting both $J$ and $L$.
By Lemma \ref{lemma LJzparallel}, since $L \in \s{M}$ and $J\in \s{L}$, we have $M\|_x J$ for some point $x \in L$.
Let $K$ be the line of $\U$ through $x$ and $y$. Then $K$ meets both $J$ and $L$.

Suppose there is a line $K' \neq K$ through $y$ meeting both $J$ and $L$.
Let $x'$ be the intersection of $K'$ and $L$.
By Theorem \ref{thm LMN}, $K'$ meets $N$, and $K$ meets $N$.
By Theorem \ref{thm LMN}, $M \|_{x} N$. Since $M \|_{x} N$ and $M\|_x J$, we have $J\|_{x}N $ and so the line $x.(J\cap K')$ meets $N$.
($J\cap K'$ denotes the intersection point of $J$ and $K'$.)
Then the four lines $K$, $K'$, $N$, $x.(J\cap K')$ form an O'Nan configuration.
This contradicts (I).

(2) $|\m_L \setminus \{L,M,N\}| = (n^2-n+1)-3 =(n + 1)(n - 2)$. This is also the number of circles of type $\C^y$ tangent to $\arc L$ in $\I(y)$. Since we have \eqref{itemSLinIy1}, to show \eqref{itemSLinIy2}, it suffices to show $\arc {J_1} \neq \arc {J_2}$ for any distinct $J_1,J_2 \in \m_L \setminus \{L,M,N\}$. Suppose not.
Then $J_1 \|_y J_2$. By (1), there is a line $K''$ through $y$ meeting $J_1,J_2,L$.
Let $x''$ be the intersection of $K''$ and $L$.
Since $J_1,J_2 \in \s{L}$ and $\m_L$ is a special spread, we have $J_1 \|_{x''} J_2$ by the definition of special spread.
Let $K'''$ be the line through $y$ and a point of $J_2$ not on $K$.
Since $J_1 \|_y J_2$ and $J_1 \|_{x''} J_2$, the four lines $K''$, $K'''$, $J_1$, $x''.(J_3\cap K''')$ form an O'Nan configuration. A contradiction.
\end{proof}

Using Lemma \ref{lemma SLinIy} and the characterization of flocks in terms of bundles in even order inversive planes by Dembowski and Hughes \cite{DH} (see also (6.2.11), (6.2.12) with footnote on p.267, and (6.2.13), of \cite {Dem2}), we deduce Lemma \ref{lemma FMN}. Lemma \ref{lemma FMN} describes when lines of $\s{L}$ meet two intersecting lines which do not belong to $\s{L}$ and which miss $L$.

\begin{lemma}\label{lemma FMN}
Let $L$ be a line of $\U$.
Let $M,N$ be distinct lines not in $\m_{L}$.
Suppose $M$ and $N$ intersect at a point $y$ of $\U$.
If $\arc{L}$ belongs to $\F(M,N)$ in $\I(y)$, then the following statements hold:
\begin{enumerate}
 \item\label{itemFMN2} Let $L_1$ be the line in $\s{L}$ through $y$. In $\I(y)$, $L_1$ is a point on the circle $C$ determined by $M$, $N$ and $\infty_y$.
 \item\label{itemFMN1} In $\U$, every line in $\s{L}$ meeting $M$ meets $N$ as well.
 \item\label{itemFMN4} In $\U$, there is a unique point $x$ on $L$ such that any line from $x$ meeting $L_1$ misses all other lines in $\s{L}$ that meet both $M$ and $N$. Furthermore, in $\I(y)$, the unital line $x.y$ through $x$ and $y$ is a point on $C$.
\end{enumerate}
\end{lemma}

\begin{figure}[!ht]
\centering
 \includegraphics[height = 4cm]{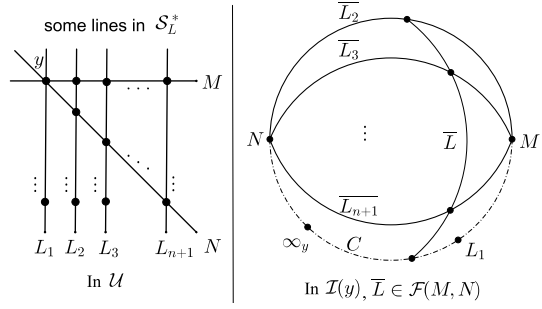}
 \caption{illustration of Lemma \ref{lemma FMN}}
\end{figure}

\begin{proof}
By Dembowski and Hughes \cite{DH}, in $\I(y)$, the flock $\F(M,N)$ is the set of circles tangent to every circle in the bundle $\B(M,N)$, i.e. the set of circles in $\I(y)$ through the points $M$ and $N$. Suppose $\arc{L}$ is in $\F(M,N)$ in $\I(y)$. Then $\arc{L}$ is tangent to every circle in $\B(M,N)$. By Lemma \ref{lemma SLinIy}, $\B(M,N) = \{\arc{L_2},\arc{L_3},\cdots,\arc{L_{n+1}},C \}$ for some distinct lines $L_2,L_3,\cdots,L_{n+1} \in \s{L}$. Since $\s{L}$ is a spread, $L_1$ does not meet $L_2,L_3,\cdots, L_{n+1}$. Hence in $\I(y)$, $L_1$ is not on $\arc{L_2},\arc{L_3},\cdots,\arc{L_{n+1}}$, but on the remaining circle $C$ of $\B(M,N)$. This proves \eqref{itemFMN2}. The lines in $\s{L}$ meeting $M$ are $L_1,L_2,\cdots, L_{n+1}$, and they meet $N$. This proves \eqref{itemFMN1}.

Let $K$ be the unital line on $y$ such that $K$ is the intersection point of $C$ and $\arc{L}$ in $\I(y)$.
Since $K$ is on $\arc{L}$ in $\I(y)$, $K$ meets $L$ at a point, say $x$, in $\U$.
Since $K$ is not on $\arc{L_2},\arc{L_3},\cdots,\arc{L_{n+1}}$ in $\I(y)$, $K$ does not meet $L_2,L_3,\cdots, L_{n+1}$ in $\U$.
Since $\s{L}$ is a special spread and $K$ is a line from $x$ meeting $L_1\in \s{L}$ but not $L_2,L_3,\cdots, L_{n+1}\in \s{L}$, $x$ is a point on $L$ satisfying the property in \eqref{itemFMN4}. Uniqueness of $x$ follows from that fact that in $\I(y)$, $K$ is the unique point on $\arc{L}$ not on $\arc{L_2},\arc{L_3},\cdots,\arc{L_{n+1}}$. This proves \eqref{itemFMN4}. 
\end{proof}

We introduce the notion of {\it triply ruled set} for a general unital: in a unital of order $m$, a set of $(m + 1)^2$ points is {\it triply ruled} if there are three partitions of the $(m + 1)^2$ points by lines. The following lemma suggests a method to find a triply ruled set.

\begin{lemma}\label{lemma parallel class}
Let $L$ and $M_1$ be disjoint lines of $\U$ with $M_1 \notin \m_L$.
Let $L_1$, $L_2$, $\cdots$, $L_{n + 1}$ be the lines of $\s{L}$ meeting $M_1$.
Then there are lines $M_2$, $M_3$, $\cdots$, $M_{n + 1}$, $N_1$, $N_2$, $\cdots$, $N_{n + 1} $ such that $\{ M_1,M_2,\cdots,M_{n + 1}\}$ and $\{ N_1,N_2,\cdots,N_{n + 1}\}$ are partitions of
the set of points covered by the disjoint lines $L_1,L_2,\cdots,L_{n + 1}$.
\end{lemma}

\begin{proof}
Let $y$ be a point on $M_1$.
Since $M_1$ and $L$ are disjoint, the point $M_1$ is not incident with $\arc{L}$ in $\I(y)$.
Since $M_1 \notin \s{L}$, condition ($p$) implies $L \notin \s{M_1}$ and hence $\arc{L} \notin \F(M_1,\infty_y)$.
Thus, there is a unital line $N_1$ through $y$ such that $\arc{L} \in \F(M_1,N_1)$.
By Lemma \ref{lemma FMN}, $L_1,L_2,\cdots,L_{n + 1}$ meet $N_1$.
For $i=2,3,\cdots,n+1$, let $y_i$ be the intersection point of $N_1$ and $L_i$.
By a similar argument, for each $i=2,3,\cdots,n+1$, there is a unital line $M_i$ through $y_i$ such that $\arc{L} \in \F(M_i,N_1)$,
and there is a unital line $N_i$ through $y_i$ such that $\arc{L} \in \F(M_i,N_i)$.
By Lemma \ref{lemma FMN}, $L_1,L_2,\cdots,L_{n + 1}$ meet $M_i,N_i$, for all $i=2,3,\cdots,n+1$.

We are going to show that $M_1,M_2,\cdots,M_{n + 1}$ are mutually disjoint. Suppose $M_{i_1}$ and $M_{i_2}$ intersect at a point $z$ for some distinct $i_1,i_2 \in \{1,2,\cdots,n + 1\}$. Then $z$ is on $L_{i_3}$ for some $i_3 \in \{1,2,\cdots,n + 1\} \setminus \{i_1,i_2\}$. Let $i_4 \in \{1,2,\cdots,n + 1\} \setminus \{i_1,i_2,i_3\}$. Then the four lines $M_{i_1}$, $M_{i_2}$, $N_1$, $L_{i_4}$ form an O'Nan configuration. A contradiction.
Hence $\{ M_1,M_2,\cdots,M_{n + 1}\}$ is a partition of
the set of points covered by $L_1,L_2,\cdots,L_{n + 1}$.
By a similar argument, the same conclusion can be drawn for $N_i$'s.
\end{proof}

Note that there cannot be a forth partition of the set of points covered by $L_1$, $L_2$, $\cdots$, $L_{n + 1}$; otherwise, there would be an O'Nan configuration constituted by lines of different partitions. This suggests the following notion of parallelism:

\begin{itemize}
\item[]Let $M$, $N$ be lines missing $L$ and not in $\m_{L}$. We say that $M$ and $N$ are {\it $L$-parallel}, denoted by $M \|_{L} N$, if they are identical or they are non-intersecting and meet the same lines in $\s{L}$. We say that $M$ and $N$ are {\it $L$-non-parallel} if they intersect and meet the same lines in $\s{L}$.
\end{itemize}

\begin{lemma}\label{lemma equiv relation}
Let $L$ be a line of $\U$. Then $\|_L$ defines an equivalence relation in the set of lines missing $L$ and not in $\m_{L}$. Each class has $n + 1$ lines. There are $n(n - 1)(n - 2)$ equivalence classes under the equivalence relation $\|_L$.
\end{lemma}

\begin{proof}
It is clear that $\|_{L}$ is reflexive and symmetric. It is transitive because of Lemma 3.3 and the fact that two points determine a line. Since non-equal $L$-parallel line are non-intersecting, there are at most $n + 1$ lines in an equivalence class. By Lemma \ref{lemma parallel class}, this upper bound is achieved. The number of lines missing $L$ is $n(n - 1)(n^2-n - 1)$. Among these lines, $n(n - 1)$ are in $\s{L}$. The result follows from simple counting.
\end{proof}

By Lemma \ref{lemma FMN}, for any distinct lines $M'$ and $N'$ through $y$, $N'$ is $L$-non-parallel to $M'$ if $\arc{L} \in \F(M',N')$ in $\I(y)$. Furthermore, the set of lines which are $L$-non-parallel to $M'$ is an $L$-parallel class.
In the context of Lemma \ref{lemma parallel class}, it is natural to ask whether $M_1,M_2,\cdots,M_{n + 1}$ are in $\s{M}$ for some line $M$. The answer is yes:

\begin{lemma}\label{lemma concurrent}
Refer to the set-up in Lemma \ref{lemma parallel class}.
Let $M$ be the line in $\s{L}$ such that $M_1\in \s{M}$.
Then for each $i=2,3,\cdots,n+1$, $M_i$ belongs to $\s{M}$.
\end{lemma}

\begin{proof}
By Lemma \ref{lemma LJzparallel}, there is a point $z_1\in M$ such that $L \|_{z_1} M_1$.
Let $z_2,z_3,\cdots,z_{n + 1}$ be the points on $M$ other than $z_1$. Let $M_{z_1} = M_1$.
Let $N$ be a line such that $L,M,N$ form a self-polar triangle.
For $i = 2,3,\cdots, n + 1$, we claim that
{\it there is a line $M_{z_i} \in \s{M}$ distinct from $L$ and $N$, such that $M_{z_i}$ is $z_i$-parallel to $L$, and $M_{z_i}$ meets $L_1,L_2,\cdots,L_{n + 1}$.}
It the claim is true, then $\{M_{z_i} | z_i \in M \}$ is the $L$-parallel class containing $M_{z_1}$.
Indeed, if $M_{z_i} = M_{z_j}$ for some $i\neq j$, then $M_{z_i}$, $L$ and $N$ would be three lines both $z_i$-parallel and $z_j$-parallel, giving an O'Nan configuration.
Hence $M_{z_i} \neq M_{z_j}$ for distinct $i,j$. Furthermore, for $i = 2,3,\cdots,n + 1$, since $M_{z_i}$ is $z_i$-parallel to $L$, $M_{z_i}$ is in $\s{M}$ by Lemma \ref{lemma LJzparallel}. Thus $M_{z_i}$'s are non-intersecting. Hence $\{M_{z_i} | z_i \in M \}$ is the $L$-parallel class containing $M_{z_1}$.
By uniqueness of such a class, $\{M_1,M_2,\cdots, M_{n+1} \}=\{M_{z_i} | z_i \in M \}$.
The result follows.

We now prove the claim. Let $k \in \{2,3,\cdots, n + 1\}$. For $i = 1,2,\cdots,n + 1$, let $K_i = (M_1 \cap L_i).z_1$. Since $L \|_{z_1} M_1$, we can label the points on $L$ as $x_1,x_2,\cdots,x_{n + 1}$ such that $x_i \in K_i$. Since $\m_L$ is a special spread, and $K_i$ passes through $x_i \in L$ and meets both $M\in\s{L}$ and $L_i\in\s{L}$, we have $M \|_{x_i} L_i$.
For $i = 1,2,\cdots,n + 1$, let $K_i' = z_k.x_i$. Since $M \|_{x_i} L_i$ and $K_i'$ meets $M$, we conclude that $K_i'$ meets $L_i$, say at $w_i$. If $w_1,w_2,\cdots,w_{n + 1} $ are collinear, then take $M_{z_k}$ to be the line that they are on and we find $M_{z_k}$ (Figure \ref{lemma35}).

\begin{figure}[!ht]
 \centering
 \includegraphics[height = 4.5cm]{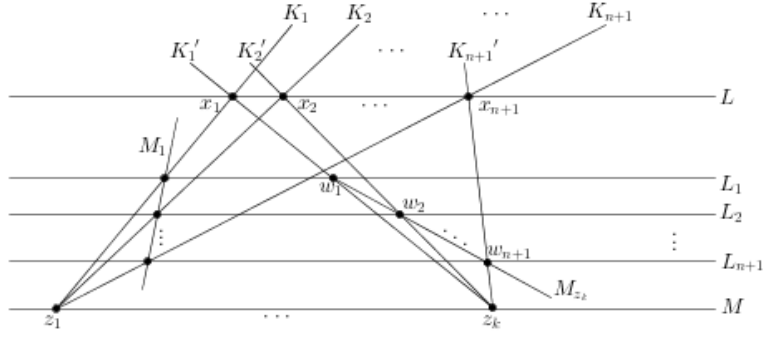}\label{lemma35}
\caption{$M_{z_k}$ is $L$-parallel to $M_1$}
\end{figure}

It remains the case when some of $w_1,w_2,\cdots,w_{n + 1}$ are not collinear. For $i = 1,2,\cdots,n + 1$, let $J_i$ be the line on $w_i$ that is $z_k$-parallel to $L$. At least two of $J_1,J_2,\cdots,J_{n + 1}$ are the same line. For, otherwise, there would be $n + 1$ lines $z_k$-parallel to $L$. Without loss of generality, assume $J_1 = J_2$.
We are going to prove that $J_1$ meets $L_1,L_2,\cdots,L_{n + 1}$ by contradiction.
Suppose $J_1$ misses some of $L_3,L_4,\cdots,L_{n + 1}$.
By Lemma \ref{lemma parallel class}, there are lines $M_2$ and $N_2$ on $w_1$ respectively $L$-parallel to $M_1$ and $L$-non-parallel to $M_1$.
Thus, in $\I(w_1)$, $\B(M_2,N_2) = \{\arc{L_2},\arc{L_3},\cdots,\arc{L_{n + 1}}, C\}$ for some circle $C$ of type $\C_{w_1}$.
Since $L_1$ are disjoint from $L_2,L_3,\cdots,L_{n+1}$ in $\U$, $L_1$ is a point on $C$ in $\I(w_1)$.
Since $J_1 = J_2$ meets $L_2$ in $\U$, $J_1$ is a point on $\arc{L_2}$ in $\I(w_1)$.
Since $J_1 \neq M_2$ and $J_1 \neq N_2$ by hypothesis, and $J_1\in \arc{L_2}$, we conclude $J_1 \notin C$.
On the other hand, since $J_1\|_{z_k} L$, $\B(K_1',J_1) = \{\arc{K_2'},\arc{K_3'},\cdots,\arc{K_{n + 1}'}, D\}$ for some circle $D$ of type $\C_{w_1}$.
Note that $L_1 \in D$. Indeed, if $L_1$ is not on $D$, then $L_1$ meets $K_i'$ in $\U$ for some $i \neq 1$.
By Theorem \ref{thm LMN}, $K_i'$ meets $N$.
Since $K_i'$ meets both $N \in \s{L}$ and $L_1\in \s{L}$, we have $N \|_{x_i}L_1$ because $\m_{L}$ is a special spread.
Thus the four lines $K_1'$, $K_i'$, $L_1$, $x_i. (K_1 \cap N)$ form an O'Nan configuration. Hence $L_1 \in D$.
Since $J_1\in D$ but $J_1\notin C$, we have $C\neq D$.
Since $\infty_{w_1}$ and $L_1$ are points on both $C$ and $D$, $K_1' \notin C$.
Then $K_1'$ meets $L_j$ in $\U$ for some $j \neq 1$.
By Theorem \ref{thm LMN}, $K_i'$ meets $N$.
Since $K_1'$ meets both $N\in \s{L}$ and $L_j \in \s{L}$, we have $N \|_{x_1} L_j$ and so the four lines $K_1'$, $K_j'$, $L_j$, $x_1. (K_j'\cap N)$ form an O'Nan configuration. A contradiction. Thus $J_1$ meets $L_1,L_2,\cdots,L_{n + 1}$. Take $M_{z_k}$ to be $J_1$ and we prove our claim.
\end{proof}

\begin{theorem}\label{thm bigP}
Let $L,M,N$ be a self-polar triangle with respect to $\U$. Then $\m_L \setminus \{L,M,N\}$, $\m_M \setminus \{L,M,N\}$, $\m_N \setminus \{L,M,N\}$ can be respectively partitioned into $n - 2$ subsets
$$\{L^1_{1},L^1_{2},\cdots,L^1_{n + 1}\}, \cdots, \{L^{n - 2}_{1},L^{n - 2}_{2}, \cdots, L^{n - 2}_{n + 1}\};$$ $$\{M^1_{1},M^1_{2},\cdots,M^1_{n + 1}\}, \cdots, \{M^{n - 2}_{1},M^{n - 2}_{2}, \cdots, M^{n - 2}_{n + 1}\};$$
$$\{N_1^1,N^1_{2}, \cdots, N^1_{n + 1}\}, \cdots, \{N^{n - 2}_{1},N^{n - 2}_{2}, \cdots, N^{n - 2}_{n + 1}\};$$
each of cardinality $n + 1$, such that for $i = 1,2,\cdots,n - 2$, the sets of points incident respectively on the lines of
$\{L^i_1,L^i_{2},\cdots,L^i_{n + 1}\}$, $\{M^i_{1},M^i_{2}$, $\cdots$, $M^i_{n + 1}\}$ and $\{N^i_{1},N^i_{2},\cdots,N^i_{n + 1}\}$ are the same.
\end{theorem}

\begin{proof}
Take a line $M^1_1 \in \m_M \setminus \{L,M,N\}$.
Note that $M^1_1 \notin \s{L}$ because $ \s{L} \cap \s{M} = \{N\}$. Let $L^1_1,L^1_2,\cdots,L^1_{n + 1}$ be the lines of $\s{L}$ meeting $M^1_1$.
By Lemmas \ref{lemma parallel class} and \ref{lemma concurrent}, there is an $L$-parallel class $\{M^1_1$, $M^1_2$, $\cdots$, $M^1_{n + 1}\}$$\subset \s{M}$ and its $L$-non-parallel class $\{N^1_1,N^1_2$, $\cdots$, $N^1_{n + 1}\}$ both partitioning the set of points covered by $L^1_1$, $L^1_2$, $\cdots$, $L^1_{n + 1}$.
Since $\{N^1_1,N^1_2$, $\cdots$, $N^1_{n + 1}\}$ is an $L$-parallel class and an $M$-parallel class, it is a subset of $\s{N'}$ where $N' \in \s{L} \cap \s{M}$ by applying Lemma \ref{lemma concurrent} twice. Hence $N' = N$.
Repeat the process $n-3$ times by taking a line $M^i_1 \in \m_M \setminus ( \{L,M,N\} \cup \{ M^k_l | k = 1,2,\cdots,i-1,l = 1,2,\cdots,n + 1\})$. This finishes the proof.
\end{proof}

\begin{remark}
We may interpret Theorem \ref{thm bigP} as follows. If $\U$ is embedded in a projective plane $\pi$ as a polar unital via the construction of \cite[Theorem 1.1]{HW2}, then $\s{L}$, $\s{M}$, $\s{N}$ are respectively the set of lines through the pole of $L$, $M$, and $N$. Thus, Theorem \ref{thm bigP} says that in $\pi$, the set of points of $\U$ is partitioned into a self-polar triangle, and $n - 2$ subsets of $(n+1)^2$ points triply ruled by lines through the vertices of the triangle.
\end{remark}

\begin{remark}
In the setting in Theorem \ref{thm bigP}, for any disjoint index sets $I_L, I_M, I_N$ such that $I_L \cup I_M \cup I_N = \{1,2,\cdots,n + 1\}$, the set $\{L,M,N\}\cup$$\{L^i_j \mid i\in I_L, j = 1,2,\cdots,n + 1\}$$\cup \{M^i_j \mid i\in I_M, j = 1,2,\cdots,n + 1\}$$\cup \{N^i_j \mid i\in I_N, j = 1,2,\cdots,n + 1\}$ is a spread of $\U$. These spreads are the subregular spreads studied by Dover \cite{Dov}.
\end{remark}

\section{From a special spread of a unital to a regular spread of $\PG(3,n)$} \label{st 4}
\noindent
From now on, we fix a line $L$, and let $x_1,x_2,\cdots, x_{n+1}$ be the points on $L$.

Following Wilbrink \cite{Wil}, we are going to construct a generalized quadrangle $GQ(L)$, which is isomorphic to $Q(4,n)$ \cite{Wil}.
We will then embed $Q(4,n)$ into $\PG(4,n)$ and choose a 3-dimensional projective space $\Sigma$ in $\PG(4,n)$. It turns out that the special spread $\m_{L}$ of $\U$ introduced in Section \ref{st 2} defines a regular spread $\S$ of $\Sigma$ (Theorem \ref{thm regular}). Using the Bruck-Bose construction \cite{BB1,BB2}, we will construct a projective plane in Section \ref{st 6} using this regular spread $\S$, such that $\U$ is embedded in a way into this projective plane as a classical unital.

To construct $GQ(L)$, we recall the definition of the sets $\A_{ij}, 1 \leq i, j \leq n + 1$ of $\U$ (\cite{Wil}, also see \cite{HW2}).
Considering $\mathcal I(x_1)$, denote the circles in the bundle $\B(L, \infty_{x_1})$ by $\{L, \infty_{x_1}\} \cup \A_{1j}$, where $j = 1, 2, \cdots, n + 1$. We have defined $\A_{11}, \A_{12}, \cdots, \A_{1, n + 1}$. Next, for each $j \in \{1, 2, \cdots, n + 1\}$, consider the pencil $<L, \{L, \infty_{x_1}\} \cup \A_{1j}>$ in $\I(x_1)$, i.e. the maximal set of mutually tangent circles through $L$ with a member the circle $\{L, \infty_{x_1}\} \cup \A_{1j}$.
For $k = 1, 2, \cdots, n - 1$, denote by $C_{jk}$ the remaining circles in the pencil. For $i \in \{2, 3, \cdots, n + 1\}$, consider the $n - 1$ lines on $x_i$ which correspond respectively to these $n - 1$ circles $C_{jk}$'s. Denote this set of lines by $\A_{ij}$. We have defined $\A_{2j}, \A_{3j}, \cdots, \A_{n + 1,j}$, for $j = 1, 2, \cdots, n + 1$. The definition of $\A_{ij}$ is independent on the choice of the point $x_1 \in L$.

The set of points of $GQ(L)$ is
$$\{ \A_{ij} \mid i,j = 1,2 \cdots n + 1 \} \cup \{\ y \mid y \mbox{ is a point of }\U\mbox{ not on }L\}.$$
The set of lines of $GQ(L)$ is
$$\{ A_i \mid i = 1,2 \cdots n + 1 \} \cup \{ B_i \mid i = 1,2 \cdots n + 1 \} \cup \{ K \mid K\neq L\mbox{ is a line of }\U\mbox{ meeting }L\}.$$
The incidence of $GQ(L)$ is as follows. $\A_{ij}$ is incident with $A_k$ if and only if $i = k$; $\A_{ij}$ is incident with $B_k$ if and only if $j = k$; for any line $K$ of $\U$ meeting $L$, $\A_{ij}$ is incident with $K$ if and only if $K \in \A_{ij}$; a point $y$ of $\U$ is never incident with $A_i$ or $B_j$ for $i,j = 1,2,\cdots,n + 1$; incidence between a point and a line of $\U$ is the natural incidence.

Consider a parabolic quadric $\P$ in $\PG(4,n)$. The points and lines of $\P$ form a generalized quadrangle $Q(4,n)$ \cite{PT}. By \cite{Wil}, $GQ(L)$ is isomorphic under some GQ isomorphism
\begin{equation}\label{eqn varphi}
\varphi: GQ(L) \longrightarrow Q(4,n)
\end{equation}
to $Q(4,n)$.

Consider the 3-dimensional subspace $\Sigma$ of $\PG(4,n)$ determined by the skew lines $\varphi(A_1)$ and $\varphi(A_2)$.
Then $\Sigma \cap \P = \{ \varphi(\A_{ij}) \mid i,j = 1,2 \cdots n + 1 \}$, and is a hyperbolic quadric $\h$ with regulus
\begin{equation}\label{eqn R0}
\R_0 = \{ \varphi(A_i) \mid i = 1,2,\cdots,n + 1 \}
\end{equation}
and opposite regulus $ \{ \varphi(B_i) \mid i = 1,2,\cdots,n + 1 \}$.
$\h$ defines a polarity
\begin{equation}\label{eqn alpha}
\alpha: \Sigma\longrightarrow \Sigma
\end{equation}
of $\Sigma$.

The tangent spaces of $\P$ are concurrent at a point {\it nucleus} ${\bf N}$ of $\P$.
Let
\begin{equation}\label{eqn mu}
\mu: \P \longrightarrow \Sigma
\end{equation}
be the function defined as follows: for any point ${\bf V}$ of $\P$, $\mu({\bf V})$ is the intersection point of $\Sigma$ and the line joining ${\bf N}$ and ${\bf V}$.
Since $\h\cap \Sigma = \h$, $\mu$ is identity on $\h$. Since $|\P| = n^3 + n^2 + n + 1 = |\Sigma|$, $\mu$ is a bijection. Hence, we have a 1-1 correspondence between points of $GQ(L)$ and that of $\Sigma$ via the composition function $\mu \varphi$. Furthermore, some quadratic cones in $\P$ are mapped to planes of $\Sigma$ because $n$ is even:

\begin{lemma}\label{lemma projcone}
Let ${\bf V}\in \P \setminus \Sigma$.
Let $\Q$ be the quadratic cone 
formed by the intersection of $\P$ and the tangent space of $\P$ at ${\bf V}$.
Then $\mu$ maps $\Q$ onto the plane $\alpha(\mu(\bf V))$ in $\Sigma$. Furthermore, the plane $\alpha(\mu(\bf V))$ meets $\h$ in an irreducible conic with nucleus $\mu(\bf V)$.
\end{lemma}

\begin{proof}
Since ${\bf V}\notin \Sigma$, every generator of $\Q$ meets $\Sigma$ in a unique point (of $\h$). Hence, for any generator $l$ of $\Q$, $\mu(l)$ is tangent to $\h$ and passes through $\mu(\bf V)$. By a property of hyperbolic quadric of even order, $\{ \mu(l) \mid l$ is a generator of $\Q\}$ is on the plane $\alpha(\mu(\bf V))$. Since $|\Q| = n^2 + n + 1$, the image set $\mu(\Q)$ is a plane. Furthermore, $\mu(\bf V)$ is the nucleus of the irreducible conic formed by the intersection of $\mu(\Q)$ and $\h$.
\end{proof}

Using Lemma \ref{lemma projcone} and the GQ isomorphism $\varphi$, we prove that every line of $\s{L}$ is mapped to a line of $\Sigma$ under $\mu \varphi$:

\begin{lemma}\label{lemma Mvarphimu}
Let $M$ be a line of $\U$ in $\s{L}$.
Then
$\mu(\varphi(M))$ is a line in $\Sigma$.
\end{lemma}

\begin{proof}
By Theorem \ref{thm LMN}, $L \|_{z_i} M$ for $i = 1,2,\cdots,n + 1$.
Let $\Q_i$ be the quadratic cone 
formed by the intersection of $\P$ and the tangent space of $\P$ at $\varphi(z_i)$.
Hence $\varphi(M) \subset \Q_i$.
By Lemma \ref{lemma projcone}, $\mu(\varphi(M)) \subset \bigcap^{n + 1}_{i = 1} \alpha(\mu(\varphi(z_i)))$ and $\alpha(\mu(\varphi(z_i)))$'s are planes in $\Sigma$. Since $| \varphi(M) | = n + 1$ and $\mu(\varphi(M))$ is in the intersection of $n + 1$ planes, $\mu(\varphi(M))$ is a line.
\end{proof}

Since $\m_L$ is a spread of $\U$ and $\mu$ is a bijection, the set $\{\mu(\varphi(L')) \mid L'\in \s{L}\}$ consists of disjoint lines. Let
\begin{equation}\label{eqn S}
\S = \R_0 \cup \{\mu(\varphi(L')) \mid L'\in \s{L}\}.
\end{equation}
Then $\S$ is a spread. We claim that $\S$ is regular (Theorem \ref{thm regular}). The justification of this claim requires the notion of tube \cite{CK}:

When $q$ is even, a {\it tube} in $\PG(3,q)$ is a pair $\mathcal T = \{l,\mathcal B\}$, where $\{l\} \cup \mathcal B$ is a collection of mutually disjoint lines of $\PG(3,q)$ such that for each plane $\Pi$ of $\PG(3,q)$ containing $l$, the intersection of $\Pi$ with the lines of $\mathcal B$ is a hyperoval. For any mutually skew lines $l_1,l_2,l_3$ in $\Sigma$, denote by $\R(l_1,l_2,l_3)$ the unique regulus determined by them. According to Cameron and Knarr \cite {CK}, if $\{l,\{l_0,l_1,\cdots,l_{q + 1}\}\}$ is a tube, then the union $\bigcup_{i = 1}^{n + 1} \R(l,l_0,l_i)$ is a regular spread in $\PG(3,q)$.

\begin{lemma}\label{lemma tube}
Let $M,N$ be lines of $\U$. Suppose $L,M,N$ form a self-polar triangle. Then the pair $\mathcal T = \{ \mu(\varphi(M))$, $\{ \mu(\varphi(N)) \} \cup \R_0 \}$ is a tube in $\Sigma$.
\end{lemma}

\begin{proof}
Let $z_1,z_2,\cdots,z_{n + 1}$ be the points of $N$. By Lemma \ref{lemma projcone}, the $n + 1$ planes in $\Sigma$ containing $\mu(\varphi(M))$ are $\alpha(\mu(\varphi(z_i)))$, $i = 1,2,\cdots,n + 1$. Since $\R_0$ is a regulus of $\h$, $\alpha(\mu(\varphi(z_i)))$ meets the points covered by $\R_0$ in an irreducible conic with nucleus $\mu(\varphi(z_i))$ by Lemma \ref{lemma projcone}. The result follows.
\end{proof}

\begin{theorem}\label{thm regular}
$\S$ is a regular spread in $\Sigma$.
\end{theorem}

\begin{proof}
Let $M,N$ be lines such that $L,M,N$ form a self-polar triangle. For $i = 1,2,\cdots,n + 1$, let $\R_i = \R( \mu(\varphi(M))$, $ \mu(\varphi(N))$, $\mu(\varphi(A_i)) )$. By \cite {CK} mentioned above, the union $\bigcup_{i = 1}^{n + 1} \R_i$ is a regular spread. We are done if we show $\S = \bigcup_{i = 1}^{n + 1} \R_i$. Thus it suffices to show $\S \subset \bigcup_{i = 1}^{n + 1} \R_i $.

Let $L_1 \in \m_L \setminus \{L,M,N\}$. Since $L\in \s{M}$ and $L_1\in \s{L}$, there is a point $x_i \in L$ such that $M \|_{x_i} L_1$ by Lemma \ref{lemma LJzparallel}.
Let $K_1,K_2,\cdots,K_{n + 1}$ be the lines of $\U$ on $x_i$ meeting $M$ (and hence meeting $L_1$).
Then they meet $N$ by Theorem \ref{thm LMN}.
Hence, $\mu(\varphi(K_1))$, $\mu(\varphi(K_2))$, $\cdots$, $\mu(\varphi(K_{n + 1}))$ are lines in $\Sigma$ meeting $\mu(\varphi(M))$, $\mu(\varphi(N))$ and $\mu(\varphi(A_i))$.
Furthermore, these $n + 1$ lines are disjoint.
Indeed, the points $K_1,K_2,\cdots,K_{n + 1}$ are incident with $\arc M \in \F(L,\infty_{x_i})$ in $\I(x_i)$, where $\arc M$ is tangent to each circle in the bundle $\B(L, \infty_{x_i})$ \cite{DH}, and so we may assume $K_j\in \A_{ij}$ for $j = 1,2,\cdots n + 1$ .
Thus, $\{\mu(\varphi(K_j)) \mid j = 1,2,\cdots,n + 1\}$ is the opposite regulus of $\R_i$. Since $\mu(\varphi(L_1))$ meets every line in the opposite regulus of $\R_i$, it is in $\R_i$.
\end{proof}

\section{From a partition of $\mathcal S^{*}_L$ to a pencil of quadrics in $\PG(3,n)$}\label{st 5}
\noindent
We use the notations in Section \ref{st 4} and continue to prove that $\U$ is classical.
The key result in this section is Theorem \ref{thm pencil}. It says that the image of the partition of $\m_{L}$ in Theorem \ref{thm bigP} under $\mu\varphi$ corresponds to a pencil of quadrics of two lines and $n-1$ hyperbolic quadrics of the projective space $\Sigma$, where $\mu$, $\varphi$ and $\Sigma$ are defined Section \ref{st 4}.
To prove Theorem \ref{thm pencil}, we have to describe every regulus of the spread $\S$ defined in \eqref{eqn S}, in terms of the geometry of $\U$.
A regulus of $\S$ has two, one, or no common lines with $\R_0$. We consider these cases separately.

We first consider the reguli of $\S$ with exactly one common line with $\R_0$. They can be described using $x$-parallelism introduced in Section \ref{st 2}, where the $x$'s are the points of the line $L$.

\begin{lemma}\label{lemma C1}
Let $i \in \{1,2, \cdots, n + 1\}$ and $L_1, L_2,\cdots,L_n \in \s{L}$.
If $\R = \{ \mu(\varphi(A_i))$, $\mu(\varphi(L_1))$, $\mu(\varphi(L_2))$, $\cdots$, $\mu(\varphi(L_n))\}$ forms a regulus in $\S$, then $L_j \|_{x_i} L_k$ for any $j,k \in \{ 1,2,\cdots,n\}$.
\end{lemma}

\begin{proof}
Let $l$ be a line in the opposite regulus of $\R$.
By the definition of $\mu$ and the construction of $GQ(L)$, there is a line $K$ of $\U$ on $x_i$ such that $l = \mu(\varphi(K))$.
Since $l$ meets $\mu(\varphi(L_1)), \mu(\varphi(L_2)) ,\cdots, \mu(\varphi(L_n))$, the line $K$ meets $L_1, L_2,\cdots,L_n \in \s{L}$ in $\U$.
Since $\m_L$ is a special spread, $L_1, L_2,\cdots,L_n$ are $x_i$-parallel.
\end{proof}

Lemma \ref{lemma C2} describes the reguli of $\S$ with exactly two common lines with $\R_0$. To prove Lemma \ref{lemma C2}, we need Lemma \ref{lemma F(K,K')}.

\begin{lemma}\label{lemma F(K,K')}
Let $i_1,i_2,j_1,j_2 \in \{1,2, \cdots, n + 1\}$ with $i_1 \neq i_2$ and $j_1 \neq j_2$.
Let $K_1 \in \A_{i_1 j_1}$ and $K_2 \in \A_{i_2 j_2}$ be lines of $\U$ meeting at a point $y$ of $\U$.
Let $M$ be a line in $\s{L}$ not through $y$.
Then $\arc {M}$ is in the flock $\F(K_1,K_2)$ in $\I(y)$ if and only if there is a point $z\in M$ such that $z.x_k\in \A_{i_k j_k}$ for $k=1,2$.
\end{lemma}

\begin{proof}
Without loss of generality, assume $i_1=j_1=1$ and $i_2=j_2=2$.
Suppose $\arc {M}\in \F(K_1,K_2)$ in $\I(y)$.
For $k=1,2$, let $N_k$ be a line through $x_k$ such that $\arc{M}\in \F(K_k,N_k)$ in $\I(x_k)$.
Applying Lemma \ref{lemma FMN}\eqref{itemFMN1} respectively to $\arc{M}\in \F(K_1,K_2)$ in $\I(y)$, $\arc{M}\in \F(K_1,N_1)$ in $\I(x_1)$ and $\arc{M}\in \F(K_2,N_2)$ in $\I(x_2)$, any line of $\s{M}$ meeting $K_1$ meets $K_2$,$N_1$ and $N_2$.
Since $M \in \s{L}$, we have $L \in \s{M}$ by condition ($p$). Hence, $L$ is the line in $\s{M}$ that meet $K_1,K_2,N_1,N_2$.
By Lemma \ref{lemma FMN}\eqref{itemFMN4}, there is a point $z_k$ on $M$ such that any line from $z_k$ meeting $L$ miss all other lines in $\s{M}$ that meet $K_1,K_2,N_1,N_2$.
By uniqueness in Lemma \ref{lemma FMN}\eqref{itemFMN4}, $z_1=z_2$.
By Lemma \ref{lemma FMN}\eqref{itemFMN2} and \eqref{itemFMN4}, in $\I(x_k)$, the points $z_k.x_k$, $K_k$, $L$, $\infty_z$ are concircular. By definition of $\A_{kk}$, we have $z_k.x_k \in \A_{k k}$ for $k=1,2$.

Conversely, note that there are exactly $n-2$ circles in $\F(K_1,K_2)$ in $\I(y)$ is of type $\C^y$, and each of these circles gives one $z\neq y$ such that $z.x_1 \in \A_{11}$ and $z.x_2 \in \A_{22}$. To prove the converse, it suffices to show there are exactly $n - 2$ such $z$'s.
By definition of $\A_{k k}$, in $\I(x_1)$, lines of $\A_{22}$ correspond to circles of a pencil with carrier $L$, and lines of $\A_{11}$ correspond to points on a circle through $L$ not in that pencil. Hence, each line of $\A_{11}$ meets exactly one line of $\A_{2 2}$ in $\U$. Since $|\A_{11}|=n-1$ and there is only one line of $\A_{11}$ passing through $y$, there are exactly $n - 2$ such $z$'s.
\end{proof}

\begin{lemma}\label{lemma C2}
Let $i_1,i_2 \in \{1,2, \cdots, n + 1\}$ with $i_1 \neq i_2$.
Let $L_1, L_2,\cdots,L_{n-1}\in \s{L}$.
Suppose $\R$ is a regulus of $\S$ containing $\mu(\varphi(A_{i_1}))$, $\mu(\varphi(A_{i_2}))$, $\mu(\varphi(L_1))$, $\mu(\varphi(L_2))$, $\cdots$, $\mu(\varphi(L_{n-1}))$.
Then for any point $z_1\in L_1$, if $K_1$ is the line passing through $z_1$ and $x_{i_1}$, and if $K_2$ is the line passing through $z_1$ and $x_{i_2}$, then in $\I(z_1)$,
$$\F(K_1, K_2) = \{ \arc{L_2},\arc{L_3},\cdots,\arc{L_{n-1}},C\}$$
for some circle $C$ through $\infty_{z_1}$.
\end{lemma}

\begin{proof}
Let $z$ be a point on $L_1$.
Let $l$ be the line in the opposite regulus of $\R$ through $\mu(\varphi(z))$.
Let $z' \in L_2$ be the unital point such that $\mu(\varphi(z')) \in l$.
Consider $\alpha$ defined in \eqref{eqn alpha}.
Then $\alpha(l)$ meets $\h$ at $\mu(\varphi(\A_{i_1 j_1}))$ and $\mu(\varphi(\A_{i_1 j_2}))$ for some $j_1,j_2$ with $j_1 \neq j_2$, and $\alpha(l)$ lies on the plane $\alpha(\mu(\varphi(z')))$.
Let $\Q$ be the quadratic cone formed by the intersection of $\P$ and the tangent space of $\P$ at $\varphi(z')$.
By Lemma \ref{lemma projcone}, $\mu^{-1}(\alpha(l)) \in \Q$ and so $\varphi(\A_{i_1 j_1}),\varphi(\A_{i_1 j_2}) \in \Q$. Hence $z'$ lies on some unital lines $K_1 \in \A_{i_1 j_1}$ and $K_2 \in \A_{i_1 j_2}$. By Lemma \ref{lemma F(K,K')}, $\arc {L_2}\in \F(K_1,K_2)$ in $\I(z)$.
Similarly, $\arc{L_3},\cdots,\arc{L_{n-1}}\in \F(K_1,K_2)$. The result follows by Lemma \ref{lemma SLinIy}.
\end{proof}

Lemma \ref{lemma C0} gives a characterization of reguli of $\S$ with no common line with $\R_0$, by considering inversive planes whose blocks are defined by the reguli of $\S$. For each line $J$ of $\U$ that misses $L$ and not in $\m_{L}$, let $C(J)$ be the set of images of the $n + 1$ lines of $\s{L}$ meeting $J$ under $\mu\varphi$.

\begin{lemma}\label{lemma C0}
A set of $n+1$ lines of $\S \setminus \R_0$ is a regulus if and only if it is $C(J)$ for some line $J$ of $\U$ missing $L$ and not in $\m_{L}$.
\end{lemma}

\begin{proof}
Consider the incidence structure
\begin{equation}
\I_1 = (\S , \C)
\end{equation}
where $\C$ is the set of the reguli of $\S$. By Theorem 4.5 (iv) of Bruck \cite{Bru}, $\I_1$ is the Miquelian inversive plane of order $n$. Note that $\R_0$ is a circle of $\I_1$.
We denote by $\C_0$ the set of those circles in $\I_1$ disjoint from $\R_0$;
by $\C_1$ the set of those circles in $\I_1$ tangent to $\R_0$;
by $\C_2$ the set of those circles in $\I_1$ secant to $\R_0$.

Let $\C_0^* = \{ C(J) \mid J$ is a line of $\U$ missing $L$ and not in $\m_{L} \}$.
Now considering the incidence structure
\begin{equation}
\I_2 = (\S , (\C\setminus \C_0) \cup \C_0^* ).
\end{equation}

By Theorem 2 of \cite{H}, provided that $\I_2$ is an inversive plane of order $n$, we will have $\I_1 = \I_2$ and thus $\C_0=\C_0^*$. Hence, to prove Lemma \ref{lemma C0}, {\it it suffices to prove that $\I_2$ is a $3$--$(n^2 + 1,n + 1,1)$ design.}

\smallskip

Since $\S$ has $n^2 + 1$ lines, $\I_2$ has $n^2 + 1$ points.

A block in $\C_0^*$ has exactly $n + 1$ points because every line which is not in $\m_{L}$ and which misses $L$ meets exactly $n + 1$ lines of $\s{L}$. Other blocks of $\I_2$ has exactly $n+1$ points because every regulus of $\S$ consists of $n+1$ lines.

$|\C_0^*| = n(n - 1)(n - 2)/2$ because there are $n(n - 1)(n - 2)$ $L$-parallel classes by Lemma \ref{lemma equiv relation}, and any class and its $L$-non-parallel class define a same block of $\I_2$. Since $\I_1$ is an inversive plane, $|\C\setminus \C_0|=n^2(n+3)/2$. Thus $\I_2$ has $n (n^2 + 1)$ blocks.

It remains to show that any two distinct blocks of $\I_2$ have at most two common points. Since $\I_2$ only differs from the inversive plane $\I_1$ in blocks missing $\R_0$, we only need to consider the case when one of the two blocks belongs to $\C_0^*$.
Let $C(J)$ be a block in $\C_0^*$.
Let $\mu(\varphi(L_k)),k=1,2,3$ be distinct points of $C(J)$, where $L_1,L_2,L_3 \in \s{L}$.

Suppose $\mu(\varphi(L_k)),k=1,2,3$, are on a block in $\C_1$.
By Lemma \ref{lemma C1}, $L_1$, $L_2$, $L_3$ are $x$-parallel for some $x$ on $L$. Since $J$ meets $L_1,L_2,L_3$ but does not pass though $x$, there is an O'Nan configuration, contradicting (I).

Suppose $\mu(\varphi(L_k)),k=1,2,3$, are on a block in $\C_2$ which contains $\mu(\varphi(A_i))$ and $\mu(\varphi(A_j))$.
Let $z$ be a point intersection of $J$ and $L_1$.
Let $K_i, K_j$ be the lines through $z$ that pass through $x_i$ and $x_j$ respectively.
By Lemma \ref{lemma C2}, $\arc {L_2}, \arc {L_3}$ are distinct circles in $\F(K_i, K_j)$ in $\I(z)$. Since circles in a flock are disjoint, $\arc {L_2}$ and $\arc {L_3}$ are disjoint. This contradicts that $J$ is a line through $z$ meeting $L_2$ and $L_3$.

Suppose $\mu(\varphi(L_k)),k=1,2,3$, are on a block $C(J_2)$ in $\C_0^*$.
Suppose $J$ and $J'$ are not $L$-parallel or $L$-non-parallel.
We claim that $L_2 \|_y L_3$ for any point $y \in L_1$. If the claim is true, then $L_1,L_2,L_3$ is a self-polar triangle by Theorem \ref{thm LMN} and so $\s{L_2} \cap \s{L_3} = \{L_1\}$.
Since $L\in \s{L_2} \cap \s{L_3}$ by condition ($p$) and $L \neq L_1$, a contradiction rises.
It follows that $J$ and $J'$ are $L$-parallel or $L$-non-parallel, and $C(J) = C(J')$.

To see that $L_2 \|_y L_3$ for any $y \in L_1$. Note that, if $y$ is a point on $L_1$, then the four lines on $y$ which are respectively $L$-parallel to $J_1$, $L$-non-parallel to $J_1$, $L$-parallel to $J_2$ and $L$-non-parallel to $J_2$, meet both $L_2$ and $L_3$. Thus by Wilbrink \cite[Lemma 1]{Wil}, $L_2 \|_y L_3$.

Hence, a block in $\C_0^*$ has at most two common points with any other block. Thus, $\I_2$ is an inversive plane of order $n$.
\end{proof}

To prove the main theorem in this section, we need Lemma \ref{lemma C0} and that fact that when $q$ is even, if a regular spread of $\PG(3,q)$ is partitioned into two lines and $q-1$ reguli, then the two lines and the hyperbolic quadrics containing the reguli lie in a pencil of quadrics (Hirschfeld \cite[Lemma 17.1.5, Corollary of Theorem 17.1.6]{Hir2}).

\begin{theorem}\label{thm pencil}
Let $M$, $N$ be lines in $\s{L}$.
Suppose $L,M,N$ form a self-polar triangle with respect to $\U$.
Consider the lines $L^1_{1},L^1_{2},\cdots,L^1_{n + 1}$, $\cdots$, $L^{n - 2}_{1},\cdots,L^{n - 2}_{n + 1}$, $M^1_{1},M^1_{2},\cdots,M^1_{n + 1}$ as obtained in the construction described in Theorem \ref{thm bigP}.
For $i=1,2,\cdots,n-1$, let $\h_i$ be the image set of the points on $L^i_1,L^i_{2},\cdots,L^i_{n + 1}$ under $\mu\varphi$.
Consider $\h$ is defined in Section \ref{st 4}.
Then $\{ \h_i \mid i=1,2,\cdots,n-2\}\cup \{\h\} \cup \{ \mu(\varphi(M)),\mu(\varphi(N)) \}$ is a pencil of quadrics in $\Sigma$ of two lines and $n-1$ hyperbolic quadrics.
\end{theorem}

\begin{proof}
By Lemma \ref{lemma C0}, for $i=1,2,\cdots,n-2$, the block $C(M^i_1)=\{\mu(\varphi(L^i_1))$, $\mu(\varphi(L^i_{2}))$, $\cdots$, $\mu(\varphi(L^i_{n + 1})) \}$ is a regulus. Thus $\h_i$ is a hyperbolic quadric.
By the definition of $\S$,
$\S = \Big(\bigcup_{i = 1}^{n - 2} C(M^i_1) \Big) \cup \R_0 \cup \{\mu(\varphi(M)), \mu(\varphi(N))\}$.
Since $\S$ is regular by Theorem \ref{thm regular}, the result follows by \cite{Hir2} mentioned above.
\end{proof}

\section{Completion of the proof that $\U$ is classical}\label{st 6}
\noindent
In this section, we use the notations in Section \ref{st 4} and complete the proof that $\U$ is classical. Recall from Section \ref{st 1} that we wish to show that $\U$ is isomorphic to the hyperbolic Buekenhout unital $\U'$ in $\PG(2,n^2) \cong \overline{ \pi(\S)}$ defined by $\P$ in Section \ref{st 4} under the Bruck-Bose representation \cite{BB1,BB2}. To this end, we define an isomorphism $\varphi'$ from $\varphi$, where $\varphi$ is the isomorphism between $GQ(L)$ and $Q(4,n)$ defined in \eqref{eqn mu}.

\begin{lemma}\label{lemma Jcoplanar}
Let $J$ be a line of $\U$ missing $L$ and not in $\m_L$.
Then $\varphi(J)$ lies on a plane which contains a line of $\S$.
\end{lemma}

\begin{proof}
We locate $\varphi(J)$ as a subset of an intersection of two quadratic cones, as follows.
Let $M$ be the line in $\s{L}$ such that $J\in \s{M}$.
By Lemma \ref{lemma LJzparallel}, there is a point $z \in M$ such that $L\|_z J$.
Let $\Q_0$ be the quadratic cone formed by the intersection of $\P$ and the tangent space of $\P$ at $\varphi(z)$.
By the definition of $GQ(L)$,
\begin{equation}\label{eqn 2}
\varphi(J) \subset \Q_0.
\end{equation}
By Lemma \ref{lemma projcone}, $\mu(\Q_0)\cap \h$ is a base of $\Q_0$ and its nucleus is $\mu(\varphi(z))$.
On the other hand, let $\Q_1$ be the cone with vertex $\bf N$ and a base $\mu(\varphi(J))$.
By the definition of $\mu$, we have
\begin{equation}\label{eqn 3}
\varphi(J) \subset \Q_1.
\end{equation}

We are going to show $\Q_1$ is a quadratic cone by studying $\mu(\varphi(J))$.
By Lemma \ref{lemma projcone} and \eqref{eqn 2},
$\mu(\varphi(J))$ lies on the plane $\mu(\Q_0)$.
Let $\h_1 = \{ \mu(\varphi(y)) \mid y$ is a point of $\U$ on a line of $\s{L}$ meeting $J\}$.
By Lemma \ref{lemma C0},
$\h_1$ is a hyperbolic quadric in $\Sigma$.
Note that $\mu(\varphi(J)) \subset \h_{1}$.
To see that the plane $\mu(\Q_0)$ is secant to $\h_1$,
consider the line $N$ of $\U$ such that $L,M,N$ form a self-polar triangle with respect to $\U$.
By Theorem \ref{thm LMN}, $L \|_z N$.
Thus, $\Q_0$ contains $\varphi(N)$,
and the plane $\mu(\Q_0)$ contains $\mu(\varphi(N))$.
Since $L\|_z J \|_z N$ and $N\in \s{L}$, we conclude that $N$ is disjoint from any line of $\s{L}$ that meets $J$.
Thus, $\mu(\varphi(N))$ is external to $\h_{1}$.
Thus, $\mu(\varphi(J))=\mu(\Q_0)\cap \h_1$ and the base $\mu(\varphi(J))$ of $\Q_1$ is an irreducible conic (Figure \ref{fig lemma61}).

\begin{figure}[!ht]
\centering
 \includegraphics[height = 6cm]{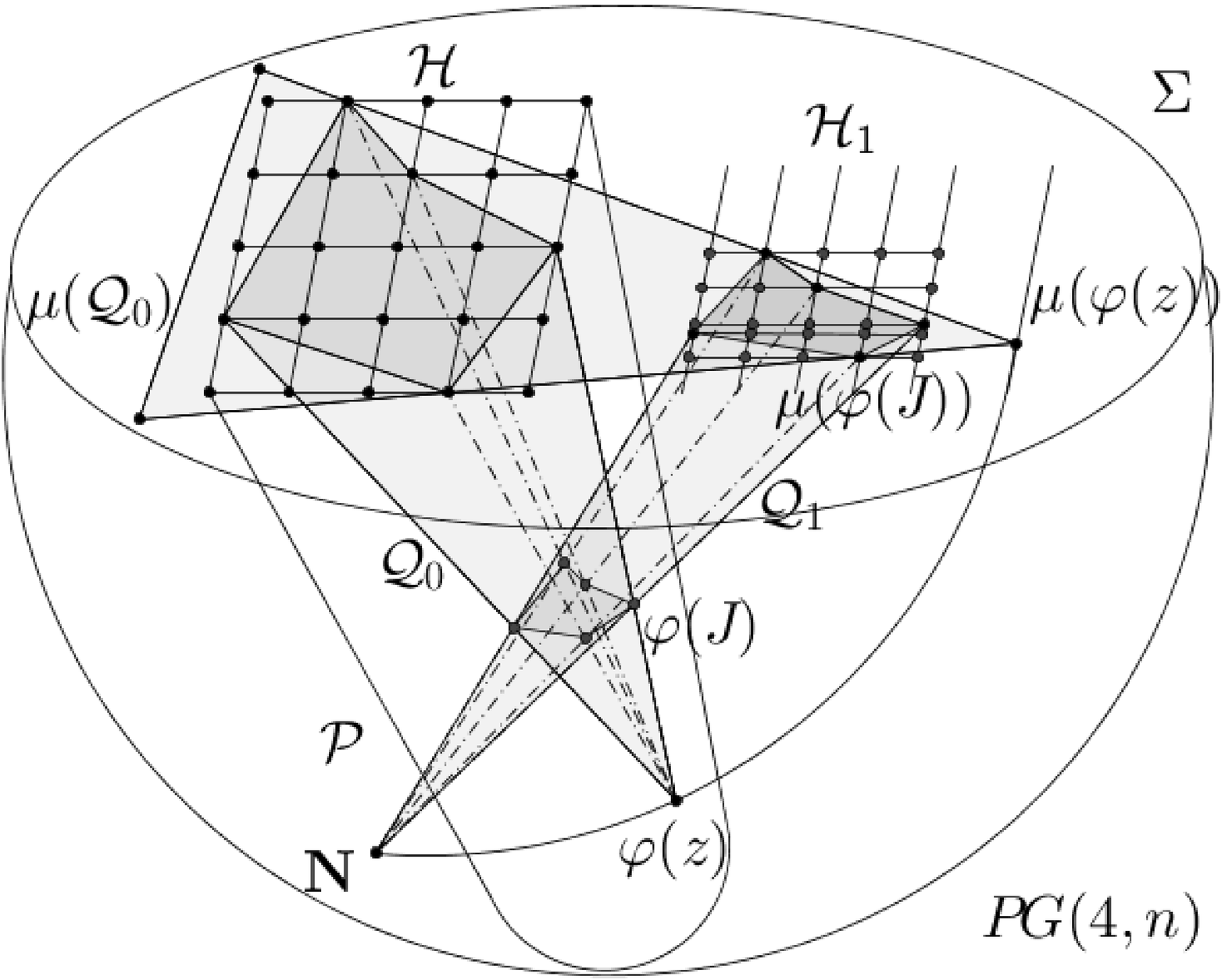}
 \label{fig lemma61}
 \caption{$\varphi(J) \subset \Q_0\cap \Q_1$}
\end{figure}

We then study $\Q_0 \cap \Q_1$.
Let $\Pi$ be the plane determined by three distinct points in $\Q_0 \cap \Q_1$.
Then $\Q_0 \cap \Pi$ and $\Q_1 \cap \Pi$ are irreducible conics.
Since $L$ is $z$-parallel to $J$, the nucleus of $\mu(\varphi(J))$ is $\mu(\varphi(z))$.
Since the nuclei of $\Q_0 \cap \mu (\Q_0)$ and $\Q_1 \cap \mu (\Q_0)$ are both $\mu(\varphi(z))$,
the conics $\Q_0 \cap \Pi$ and $\Q_1 \cap \Pi$ have a same nucleus, namely the intersection of $\Pi$ and the line through $\bf N$, $\varphi(z)$ and $\mu(\varphi(z))$.
Since $\Q_0 \cap \Pi$ and $\Q_1 \cap \Pi$ are irreducible conics with the same nucleus and containing three distinct common points, $\Q_0 \cap \Pi=\Q_1 \cap \Pi$ by Lemma 2.1 of Luyckx \cite{L}.
Thus, $\varphi(J)\subset \Pi$.

By Theorem \ref{thm pencil}, $\h_1, \h,\mu(\varphi(M)),\mu(\varphi(N))$ are quadrics in a same pencil.
Thus, $\h_{1}$ and $\h$ meet $\mu(\varphi(N))$ in the same conjugate pair of points with respect to $\mathbb F_{q^2}$ \cite[Theorem 5.1]{BH}.
Hence the bases of $\Q_0$ and $\Q_1$ pass through the conjugate pair of points in $\PG(4,q^2)$, and so do $\Q_0 \cap \Q_1$ and $\Pi$. Hence, the plane $\Pi$, where $\varphi(J)$ lies on, contains the line $\mu(\varphi(N)) \in \S$.
\end{proof}

We prove that a unital $\U$ is classical by studying the hyperbolic Buekenhout unital $\U'$ in $\overline{ \pi(\S)}$ defined by $\P$, where $\overline{ \pi(\S)}$ is the projective plane constructed by the Bruck-Bose construction \cite{BB1,BB2}
(see Andr\'{e} \cite{An} for an alternative treatment).

\begin{theorem}\label{thm main}
$\U$ is classical.
\end{theorem}

\begin{proof}
Consider the incidence structure $\pi(\S)$ whose points are the affine points of $\PG(4,q)$ and whose lines are the affine planes of $\PG(4,q)$ each containing a line of $\S$. The incidence of $\pi(\S)$ is the incidence of $\PG(4,q)$. By Bruck and Bose \cite{BB1}, $\pi(\S)$ is an affine translation plane of order $q^2$.
Complete $\pi(\S)$ to a projective plane $\overline{ \pi(\S)}$ by adding a line at infinity, $L_\infty$.
Since $\S$ is regular, $\overline{ \pi(\S)}$ is Desarguesian \cite{BB2}.
By Buekenhout \cite{Bkt}, $\P$ corresponds to a hyperbolic Buekenhout unital $\U'$ in $\overline{ \pi(\S)}$.
Let $a_i$ be the points on $L_\infty$ corresponding to the $\varphi(A_i)$'s in $\R_0$.
Then the point set of $\U'$ is $(\P \setminus \Sigma) \cup \{a_1,a_2,\cdots,a_{q + 1}\}$.
Since $\S$ is regular by Theorem \ref{thm regular}, $\U'$ is classical by Barwick \cite{Bar}.

Let $\varphi': \U \longrightarrow \U'$ be a function defined by
$\varphi'(x_i) = a_i $ for $ i = 1,2,\cdots, n + 1$;
$\varphi'(y) = \varphi(y) $ for any point $y$ of $\U$ not on $L$;
$\varphi'(J) = \{\varphi(y) \mid y \in J \}$ for any line $J$ of $\U$ missing $L$;
$\varphi'(K) = \{\varphi(y) \mid y \in K \setminus L \} \cup \{a_i\}$ for any line $K$ of $\U$ meeting $L$ at some point $x_i$;
$\varphi'(L) = \{a_1,a_2,\cdots,a_{q + 1}\}$.

Note that $\varphi'$ is a well-defined function.
Indeed, for any line $J_1 \notin \s{L}$ of $\U$ missing $L$, its image $\varphi'(J_1)$ is on a secant plane on a line of $\S$ by Lemma \ref{lemma Jcoplanar}, and hence is a line of $\U'$.
As for a line $J_2 \in \s{L}$ of $\U$ missing $L$, its image $\varphi'(J_2)$ is on the secant plane of $\P$ determined by the point $\bf N$ and the line $\mu(\varphi(J_2))\in \S$, and hence $\varphi'(J_2)$ is a line of $\U'$.
As for a line $K$ meeting $L$ at some $x_i$, its image $\varphi'(K)$ consists of $a_i$ and the $n$ affine points of $\PG(4,n)$ on a line of $\P$ meeting $\varphi(A_i)$, and hence is a line of $\U'$. Since $\varphi$ is an isomorphism, $\varphi'$ preserves incidence. Clearly, $\varphi'$ is injective. Thus, $\varphi'$ is a design isomorphism and $\U$ is classical.
\end{proof}

\begin{remark}
By \cite[Theorem 1.1]{HW2}, since $\U$ satisfies ($p$), $\U$ can be embedded in a projective plane $\pi$ as a polar unital. The author does not know whether $\pi$ is Desarguesian or not.
\end{remark}

\begin{acknowledgment}
This article is part of the author's Ph.D. thesis, written under the supervision of Philip P.W. Wong at the University of Hong Kong. The author acknowledges the suggestions of Philip Wong during the preparation of the paper. The author also wish to thank Tim Penttila and Yee Ka Tai for their suggestions of simplifying the proof of Lemma \ref{lemma Jcoplanar}.
\end{acknowledgment}

\end{document}